\documentclass[12pt]{amsart}
\usepackage{graphicx,epsfig}
\usepackage{amssymb,amsfonts,amsmath,amsthm,dsfont,wasysym,pifont,stmaryrd}
\usepackage{epstopdf,yfonts}
\usepackage{hyperref}

\usepackage[mathcal]{eucal}

\newtheorem{theorem}{Theorem}[section]
\newtheorem*{theorem'}{Theorem }
\newtheorem*{question}{Question}

\newtheorem{prop}[theorem]{Proposition}

\newtheorem{cor}[theorem]{Corollary}

\newtheorem{lemma}[theorem]{Lemma}

\renewcommand{\(}{\left(}
\renewcommand{\)}{\right)}
\renewcommand{\[}{\left[}
\renewcommand{\]}{\right]}

\renewcommand{\~}[1]{\overline{#1}}
\renewcommand{\geq}{\geqslant}
\renewcommand{\leq}{\leqslant}

\renewcommand{\>}{\right\rangle}

\newcommand{\C}{\mathbb{C}}
\renewcommand{\Cap}[2]{\underset{#1}{\overset{#2}{\cap} }}

\newcommand{\EL}{\textnormal{EL}}
\newcommand{\f}{\varphi}

\newcommand{\G}{\Gamma}

\newcommand{\GL}{\textnormal{GL}}

\renewcommand{\int}{\varint}

\newcommand{\norm}{\trianglelefteqslant}

\newcommand{\Op}{\mathfrak{O}}

\newcommand{\Q}{\mathbb{Q}}

\renewcommand{\Re}{\mathbb{R}}

\newcommand{\SL}{\textnormal{SL}}
\newcommand{\stab}{\mathrm{stab}}

\newcommand{\Z}{\mathbb{Z}}

\title[On images of real representations of $\SL_n(\Op)$]{On images of real representations of special \\ linear groups over Complete Discrete Valuation Rings} 
\author{Talia Fern\'os, Pooja Singla}
\address{Department of Mathematics and Statistics\\University of North Carolina, Greensboro}
\email{t\_fernos@uncg.edu}
\address{Department of Mathematics \\ Indian Institute of Science\\ Bangalore 560012 \\ INDIA
}
\email{pooja@math.iisc.ernet.in}
\subjclass[2010]{Primary 20G25; Secondary 20G05}
\keywords{Image of abstract homomorphism, Ring of integers, 
Real representations of special linear groups, Finite images, complete discrete valuation rings}

\begin{document}
\maketitle


\begin{abstract}
In this paper, we investigate the abstract homomorphisms of the special linear group $\SL_n(\Op)$ over complete discrete valuation rings  with finite residue field into the general linear group $\GL_m(\Re)$ over the field of real numbers. We show that for $m < 2n$, every such homomorphism factors through a finite index subgroup of $\SL_n(\Op)$. For $\Op$ with positive characteristic, this result holds for all $m \in \mathbb N$.
\end{abstract}
\section{Introduction}

Borel and Tits showed in 1973 that in ``most" cases, abstract homomorphisms between algebraic groups are in fact algebraic \cite{BorTits}, i.e. \emph{any} homomorphism $\f: G(k) \to G'(k')$ ``almost'' arises out from a field-morphism $k \to k'$.

In 1975 Margulis showed that higher rank lattices are superrigid. Employing the Borel-Harish Chandra theorems, this means that if $R$ and $k$ are a suitably chosen ring and field respectively then, any abstract homomorphism $G(R) \to G'(k)$  again \emph{almost} arises out of a ring-morphism $R \to k$. 

These results beg the following motivating question:

\begin{question}
 Let $R$ and $R'$ be rings and $G$ and $G'$ be group schemes so that $G(R)$ and $G'(R')$ are well defined. When are the homomorphisms $G(R) \to G'(R')$ dictated by ring-morphisms $R \to R'$?
\end{question}

We purposefully leave $G$ and $G'$  vaguely defined. The  reader may consider algebraic group schemes, or even the group generated by \emph{elementary unipotent } matrices over $R$, which will be defined shortly. Answering questions along these lines, we have: 

\begin{itemize}
\item \cite{BorTits} Let $k$ be an infinite field, $G$ and $G'$ be absolutely almost simple algebraic groups with $G$ simply connected or $G'$ adjoint, and $G$ generated by $k$-unipotents. Modulo the finite centers of $G$ and $G'$, any abstract homomorphism $G(k) \to G'(k')$ with Zariski-dense image arises out of a field homomorphism $k\to k'$.
\item \cite{Margulis}, \cite{BorHarChan1}, \cite{BorHarChan2} Let $\Op$ be the ring of integers of a number field $k$ and $G$ be higher rank and defined over $k$. Let $G'(\C)$ be non compact. Then, any Zariski-dense homomorphism $G(\Op) \to G'(\C)$ arises from a ring-morphism $\Op \to \C$. 
\item \cite{Fer2006} Let $n\geq 3$. Every homomorphism $\SL_n(\Z[x]) \to \GL_D\~\Q$ is not injective. This is a reflection of the fact that $\Z[x]$ does not admit a unital ring embedding into $\~\Q$. 
\item \cite{Shen} Let $n\geq 3$. Any semisimple representation $\SL_n(\Z[x_1, \dots, x_m]) \to \SL_D \C$ is virtually the direct sum of tensor products of ring homomorphisms  $\Z[x_1, \dots, x_m] \to \C$.
\item \cite{KasSap} Let $\Z\<x, y\>$ be the free non-commutative ring on $x$ and $y$. The group \linebreak  $\EL_3(\Z\<x, y\>)$ generated by elementary unipotents over the ring $\Z\<x,y\>$ does not have a faithful finite dimensional representation over any field.
\item The most recent result is due to Igor Rapinchuk \cite{Rapinchuk}. It applies to the very general context of higher rank universal Chevalley-Demazure group schemes, describing their abstract representations into $\GL_D(\mathbb K)$, where $\mathbb K$ is an algebraically closed field. We state an example which we feel both captures the essence of the result and  is relevant to our current work. Let $\Op$ be a local principal ideal ring and $n\geq 3$. Let \linebreak $\f : \SL_n (\Op) \to \GL_D(\C)$ be an abstract homomorphism. If the image is not finite then there exists a commutative $\C$-algebra $B$, an embedding $\iota:\SL_n (\Op) \to \SL_n \C$ (induced from a ring embedding $\Op \to B$) so that, up to finite index, $\f$ factors through $\iota$ composed with a $\C$-algebraic map $\SL_n \C\to \GL_D(\C)$. We remark that the general nature of this theorem makes us believe that, with some additional work, our result for $n \geq 3$ may be deduced from his. On the other hand, our inductive proof holds in the case of $n=2$, and is therefore distinct from Rapinchuk's. 
\end{itemize}


Let $\Op$ be a complete discrete valuation ring with finite residue field. The typical examples of such rings are $\mathbf Z_p$(the ring of $p$-adic integers) and $\mathbf{F}_q[[t]]$(the ring of formal power series with coefficients over a finite field).  Our main result is the following:
\begin{theorem} \label{main theorem}
For every $n \in \mathbb N$ and $D<2n$, the image of any abstract homomorphism $\f: \SL_n(\Op) \to \GL_D(\Re)$ is finite. Furthermore, if $\Op$ has positive characteristic then the image of $\f$ is finite for all $D$.  
\end{theorem}

\noindent
\textbf{Remark:} The proof of Theorem  \ref{main theorem} is completely elementary. In particular, it does not rely on  Margulis Super-Rigidity.

The connection between this result and our motivating question is as follows: in the absence of unital ring-morphisms from $\Op \to \Re$, the result means that these abstract homomorphisms are indeed, up to finite index, dictated by ring-morphisms $\Op \to \Re$. Namely, up to restricting to a finite index subgroup, they arise from the zero map $\Op \to 0\in \Re$. This interpretation is clear in the context of our proof. Our objective is to show that a sufficient amount of the ring structure can be expressed in terms of the group structure of $\SL_n$. 

Fix $x\in\Op$ and $i\neq j$. We denote the elementary unipotent matrix with 1's on the diagonal,  $x$ in the $(i,j)^{th}$-entry, and 0's elsewhere by $E_{i,j}(x)\in \SL_n(\Op)$. Consider  the following two equations:

$$[E_{1,2}(x), E_{2,3}(y)] = E_{1,3}(xy)$$
$$E_{1,3}(x)\cdot E_{1,3}(y) = E_{1,3}(x+y) $$

This shows that if $n\geq 3$ both the additive and multiplicative structures of a ring are embedded in the group structure of $\SL_n$. This is not possible for $n=2$ but there is still a sufficient amount of information that is held about the ring inside the group structure of $\SL_2$, provided the ring has many units. The task is then  to pass this information, via the homomorphism from the source to the target, which is the essence of the proof. 

A consequence of our result is that if $D < 2n$ then the $D$-dimensional real representations of $\SL_n(\Op)$, as an abstract group, are continuous in the local-topology.

\section{Algebraic Facts}\label{Preliminaries}
In this section, we give a few algebraic facts that we shall need for the proof of Theorem~\ref{main theorem}. Recall $\Op$ is a complete discrete valuation ring and therefore is a principal ideal domain with a unique maximal ideal. Let $\pi$ be a fixed generator of the maximal ideal of $\Op$. Being a discrete valuation ring, $\Op$ has a natural topology on it and we shall consider this topology on $\Op$ in the sequel. 
\begin{lemma}
\label{finite index subgroup form} 
For any $\Op$ with zero characteristic, an additive subgroup is of finite index if and only if it contains a subgroup of the form $\pi^k \Op$.
\end{lemma}
 
\begin{proof} Let $A$ be a finite index subgroup of $\Op$. Then $A$ is both open and closed as a subgroup of $\Op$. The ring $\Op$ is a finite extension of $\Z_p$ and therefore there exist $x_1, x_2, \ldots, x_g\in \Op$ which generate $\Op$ over $\Z_p$. By hypothesis $\Op/A$ finite implies that there exists an integer $m$ such that for all $1 \leq i \leq g$ the elements $m x_i$, and therefore $\Z[mx_1, mx_2, \ldots. mx_g]$, are contained in the kernel of the projection map $\Op \rightarrow \Op/A$. But then $A$ closed implies that $m \Z_p[x_1, \dots, x_g] = \pi^{ \mathrm{val}(m)}\Op$ is contained in $ A$.

\end{proof}
\begin{lemma}(Generalized Hensel's Lemma) \label{Generalized Hensel's Lemma}
\label{ghl} Let $f(x) \in \Op[x]$ be a polynomial. If there exists $a \in \Op$ such that 
$$
f(a) \equiv 0(\mathrm{mod} f'(a)^2 \pi \Op ),
$$
then there exists $a_0 \in \Op$ satisfying
$$  
f(a_0) = 0 \,\, \mathrm{and} \,\, a_0\equiv a (\mathrm{mod}f'(a) \pi \Op).
$$
If $f'(a)$ is a nonzero divisor in $\Op$, then $a_0$ is unique.
\end{lemma}  
For a proof see \cite[Theorem 2.24]{MR1083765}  

\begin{lemma}\label{nonreal} For any $\Op$ with zero characteristic, there is a positive integer $r$ and an element $q \in \Op^*$ so that $q^4 = -r$. 
\end{lemma}
\begin{proof} It is enough to prove this result for $\Z_p$ as $\Op$ is a finite extension of $\Z_p$. For $\Z_p$, the proof follows by applying Lemma \ref{Generalized Hensel's Lemma} to the following $f(x) \in \Z_p[x]$.
$$
f(x) = \left\{\begin{array}{cc} x^4 + 31, & \mbox{if} \, \, p =2; \\ 
x^4 + (p-1), & \mbox{otherwise.}
\end{array}
\right.
$$
\end{proof}

Recall that, for a ring $\mathcal{R}$ (not necessarily unital), the elementary unipotent matrices $E_{ij}(x) \in M_n(\mathcal{R})$ for $x \in \mathcal{R}$ and $i \neq j$ are the matrices with $1$'s on the diagonal,  $x$ in the $(i,j)^{th}$-entry, and $0$'s elsewhere. We denote by $\EL_n(\mathcal{R})$ the group generated by the set of elementary unipotents $\{E_{i,j}(x) \in M_n(\mathcal{R}) : x\in \mathcal{R} \text{ and } i\neq j\}$.

\break
\begin{lemma} \cite[Proposition 5.1]{MR0174604} \label{generated by unipotents} \label{EL finite index}  \begin{enumerate}
 \item  The group $\SL_n(\Op)$ is generated by elementary unipotents for $n\geq 2$.
 \item   The subgroup $\EL_n(\pi^k \Op)$ is of finite index in $\SL_n(\Op)$, for $n\geq 2$.
\end{enumerate}
\end{lemma}

\begin{cor} \label{FiniteImage}
 If $\rho: \SL_n(\Op) \to G$ is a representation so that for some $i\neq j$ the image $\rho(E_{i,j} (\Op))$ is finite then $\rho( \SL_n(\Op))$ is finite.
\end{cor}

\begin{proof}
 If the image $\rho(E_{i,j} (\Op))$ is finite, then there is some $k$ so that $E_{i,j} (\pi^k\Op)\leq \ker(\rho)$. For any $r \neq s$ with $1 \leq r,s \leq n$, the groups $E_{i,j} (\pi^k\Op)$ and $E_{r,s} (\pi^k\Op)$ are conjugate in $\SL_n(\Op)$ therefore the group $E_{r,s} (\pi^k\Op)$ is also contained in the kernel of $\rho$. This means that $\EL_n(\pi^k \Op) \leq \ker(\rho)$ and hence by Lemma~\ref{generated by unipotents} the kernel has finite index in  $ \SL_n(\Op)$.
\end{proof}

\begin{prop}\label{perfect}
Every finite index subgroup of $\SL_n(\Op)$  has finite abelianization, i.e. it is strongly almost perfect. Furthermore, if either $|\Op/\pi \Op| > 3$ or $n> 2$ then  $\SL_n(\Op)$ is perfect.
\end{prop}

\begin{proof}
Let $G \leq \SL_n(\Op)$ be a finite index subgroup. Then, for each $i, j$ with $i\neq j$ the subgroup $G \cap E_{i,j} (\Op)$ must be of finite index in $E_{i,j} (\Op)$ and hence $G \geq \EL_n(\pi^k\Op)$ for some $k$.  Therefore, it is sufficient to show that $ \EL_n(\pi^k \Op)$ has finite abelianization. 

 For $n\geq 3$ this follows from the Steinberg relations which in fact shows that both $\EL_n(\pi^k \Op)$ and  $\SL_n(\Op)$ are perfect.

For the case of $n=2$ we further subdivide to consider two cases according to whether $|\Op/\pi \Op| >3$ or $|\Op/\pi \Op| \leq 3$.

Assume $|\Op/\pi \Op| >3$. Then, there is an $\xi \in \Op^*$ such that $\xi^2-1$ is invertible.  Indeed, $ (\Op/\pi \Op)^*$ is a cylic group of order greater than 2, which means that there is an element of order greater than 2. Let $\xi$ be a lift of this element under the natural map  $\Op \to (\Op/\pi \Op)$. Then $\xi^2 - 1$ is not in the kernel $ \pi\Op$ and hence $\xi^2-1$ is invertible. Then by Lemma 1.6 \cite{MR0174604} which states that if there is $\xi\in \Op^*$ such that $\xi^2-1$ is invertible then $\SL_n(\Op)$ is perfect, we obtain our result. 

For general $|\Op/\pi\Op|$ and $n = 2$: Observe that
$$
\left( \begin{matrix} 1 & \pi^k  \\ 0  & 1 \end{matrix} \right) \left( \begin{matrix} 1 & 0 \\ \pi^k   & 1 \end{matrix} \right) = 
\left( \begin{matrix} 1+ \pi^{2k}  & \pi^k  \\ \pi^k  & 1 \end{matrix} \right).
$$
Therefore after multiplying by suitable elements of $\EL_2(\pi^k \Op)$ we see that for some $x\in \Op$ the following element belongs to $\EL_2(\pi^k \Op)$:
$$\left( \begin{matrix} 
1+\pi^{2k} x & 0  \\  
0 &  (1+\pi^{2k}x)^{-1} 
\end{matrix}
\right).
$$     

Let $q=  1+\pi^{2k}x$. Then, $q^2-1 = \pi^{k_0}x'$ for some $k_0\geq 2k$ and $x' \in \Op^*$. Therefore, the commutator subgroup of 
$\EL_2(\pi^k \Op)$ contains 
 \begin{equation}\label{comm relation}
\[\(\begin{matrix} q & 0\\ 0 &q^{-1} \end{matrix}\), \(\begin{matrix} 1 & t\\ 0 &1 \end{matrix}\)\] = \(\begin{matrix} 1 & (q^2-1)t\\ 0 &1 \end{matrix}\),\text{ for every } t\in \pi^k\Op,
\end{equation}

and in particular, contains the subgroup 
$$\left( \begin{matrix} 
1 & \pi^{k_0+k}\Op  \\  
0 & 1
\end{matrix}
\right). 
$$

Considering the transpose analogue of the commutator relation (\ref{comm relation}) we see that  the commutator subgroup of $\EL_2(\pi^k \Op)$ contains the finite index subgroup $\EL_2(\pi^{k_1} \Op)$ for  $k_1= k_0 + k$.
Hence $\EL_2(\pi^k \Op)$  has finite abelianization. 
 
\end{proof}

\begin{lemma}\label{unip}
 If $S\leq \GL_D(\Re)$ is a solvable subgroup then there exists a finite index subgroup $S_0 \norm S$ such that $[S_0, S_0]$ is unipotent and is conjugate to an upper triangular unipotent group via an element of $\GL_D(\Re)$.
\end{lemma}

\begin{proof}
 Let $S_0$ be the finite index subgroup so that the Zariski closure $\~S_0^Z(\C)$ is Zariski-connected. By the Lie-Kolchin Theorem \cite{Humphreys} $\~S_0^Z(\C)$ is conjugate into the upper triangular group and the commutator subgroup $[\~S_0^Z(\C), \~S_0^Z(\C)]$ is unipotent. This means that $[S_0, S_0] \leq [\~S_0^Z(\C), \~S_0^Z(\C)]$ is unipotent. Since the entries of $S$ are in $\Re$, there is an $\Re$-basis which upper-triangulates the unipotent group $[S_0, S_0]$. 
\end{proof}

For a ring $\mathcal{R}$ we will denote by $N_n(\mathcal{R}), U_n(\mathcal{R}),D_n(\mathcal{R})\leq \EL_n(\mathcal{R})$  the maximal upper triangular group, maximal upper  triangular  unipotent group, and the maximal diagonal group respectively. 
 
\begin{lemma}\label{FinIndU} If $N_0$ is a finite index subgroup of $N_n(\Op)$ then $U_n(\Op) \cap [N_0, N_0]$ has finite index in $U_n(\Op)$.
\end{lemma}

\begin{proof}

The proof is by induction on $n$. 
 
For $n=2$, Let $N_0 \leq N_2(\Op)$ be the finite index subgroup of interest. Observe that, since $N_0 \cap D_2(\Op)$ is finite index in $D_2(\Op)$, there is an integer $k \geq0$ such that 
$$\(\begin{matrix} 1+\pi^k  & 0\\ 0 &(1+\pi^k )^{-1} \end{matrix}\) \in N_0.$$

Similarly, $N_0 \cap E_{12}(\Op)$ has finite index in $E_{12}(\Op)$ and by Lemma \ref{finite index subgroup form}  contains $E_{12}(\pi^r\Op)$ for some positive integer $r$.

Apply the commutation relation (\ref{comm relation}) with $q= 1+\pi^k$ and $t\in \pi^r\Op$ and we see that $[N_0, N_0] \cap U_n(\Op)$ contains the finite index subgroup

$$\left( \begin{matrix} 
1 & \pi^{k+r}\Op  \\  
0 & 1
\end{matrix}
\right).
$$

Now, assume it is true for $n$ and let us show it for $n+1$. Consider $N_n(\Op) \hookrightarrow N_{n+1}(\Op)$ by taking the last column of $N_{n+1}(\Op)$ to be trivial. Similarly we have  $U_n(\Op) \hookrightarrow U_{n+1}(\Op)$. Let $N_0\leq N_{n+1}(\Op)$ be the finite index subgroup in question. And let $N_0' = N_0 \cap N_n(\Op)$.

Consider,  $[N_0', N_0'] \cap U_n(\Op)$. Then, by induction $[N_0', N_0']\cap U_n(\Op) \geq U_n(\pi^k \Op)$ for some $k\geq 0$. Observing that $[N_0, N_0] \cap U_{n+1}(\Op)$ is normal in $U_{n+1}(\Op)$ the following commutator relation gives the desired result:
$$[E_{i, n} (\pi^k \Op), E_{n, n+1} ( \Op)] = E_{i, n+1} (\pi^k\Op).$$

\end{proof}
Combining Lemmas~\ref{unip} and \ref{FinIndU}, we obtain the following: 

\begin{lemma}\label{NiceImageUnip}
Let $\f: N_n(\Op) \to \GL_D(\Re)$ be a homomorphism. Then, there exists a normal finite index subgroup $N_0$ of $N_n(\Op)$ such that $U_0=[N_0, N_0] \cap U_n(\Op)$ is of finite index in $U_n(\Op)$ and so that the image $\f(U_0)$ is unipotent. 
\end{lemma}

\section{Proof of Theorem~\ref{main theorem}} \label{Pf}
\subsection{Proof in positive characteristic}

Let $\Op$ be a complete discrete valuation ring of positive characteristic. Let $\f:\SL_n(\Op) \to \GL_D(\Re)$ be a homomorphism. We shall show that the image of $\f$ is finite. 

With Lemma \ref{NiceImageUnip} we find a finite index subgroup $U_0 \leq U$ so that $\f(U_0)$ is unipotent. The ring $\Op$ has positive characteristic implies all the elements in $U_0$, and therefore of $\f(U_0)$, have finite order. Being a unipotent subgroup of $\GL_D(\Re)$ we obtain $\f(U_0)$ is finite. By Corollary \ref{FiniteImage} we conclude that the image $\f(\SL_n(\Op))$ is finite 

%
%
%
%
%
%
%

\subsection{Proof in Characteristic Zero  }

Now onwards we assume that $\Op$ is a complete discrete valuation ring of zero characteristic. We use induction for this case. We prove this for $n = 2$ first. 

{\bf Step 1: $\SL_2(\Op) \to \GL_2(\Re)$} 

Proof in this case follows from the following proposition combined with Corollary \ref{FiniteImage}.

\begin{prop}\label{HomOfN}
 For any representation $\f: N_2(\Op) \to \GL_2(\Re)$ the image $\f(U_2(\Op))$ is finite.
\end{prop}
\begin{proof}
With the representation fixed, let $U_0 \leq U_2(\Op)$ be the finite index subgroup guaranteed by  Lemma \ref{NiceImageUnip} so that $\f(U_0)$ is unipotent. 

If the 1-eigen space of $\f(U_0)$ is 2-dimensional then the map $\f$ factors through $U_0$ and the result follows. Therefore, assume by contradiction that it is one dimensional. Since the image $\f(U_0)$ has $\Re$-entries, the 1-eigen space is defined over $\Re$ and so, up to post composing with an inner automorphism of $\GL_2\Re$ we may assume that the image $\f(U_0)$ is upper triangular unipotent.

Since the  image of the centralizer (respectively normalizer) of $U_0$ must centralize (respectively normalize) the image of $U_0$ we see that the image $\f(U_2(\Op))$ is upper triangular with $\pm 1$ on the diagonal (and respectively the image $\f(N_2(\Op))$ is upper triangular). 

This gives rise to an additive map $\psi_A: \Op \to \Re$ and multiplicative maps $\psi_i: \Op^* \to \Re^*$ as follows:

$$ \f\(\begin{matrix} 1 & x\\ 0 &1 \end{matrix}\) = \(\begin{matrix} \pm 1 & \psi_A(x)\\ 0 &\pm 1 \end{matrix}\).$$
and 
$$\f\(\begin{matrix} q & 0\\ 0 &q^{-1} \end{matrix}\) = \(\begin{matrix} \psi_1(q) & *\\ 0 &\psi_2(q) \end{matrix}\) = \(\begin{matrix} \psi_1(q) & 0\\ 0 &\psi_2(q) \end{matrix}\) \(\begin{matrix} 1 & *\\ 0 & 1 \end{matrix}\)$$ 

Consider the following relation for $q^2\in \Op^*$, $x\in \Op$, $r\in \Z$:

$$ \(\begin{matrix} q^2 &0\\ 0 & q^{-2} \end{matrix}\) 
 \(\begin{matrix} 1 & x\\ 0 & 1 \end{matrix}\)
  \(\begin{matrix}q^{-2} & 0\\ 0 & q^2 \end{matrix}\)
=    \(\begin{matrix} 1 & q^4x\\ 0 &1 \end{matrix}\) $$

Using our definitions of $\psi_i$ and $\psi_A$, after applying $\f$ to both sides of the equation above and observing that $\left( \begin{matrix} 1 & * \\ 0 & 1 \end{matrix} \right)$ centralizes the image of $U_0$ we get the following:

$$
\(\begin{matrix} \psi_1(q)^2 & * \\ 0 & \psi_2(q)^{-2} \end{matrix}\) 
 \(\begin{matrix} \pm1 & \psi_A(x)\\ 0 & \pm1 \end{matrix}\)
 \(\begin{matrix}\psi_1(q)^2 & * \\ 0 & \psi_2(q)^{-2} \end{matrix}\)
 =
 \(\begin{matrix} \pm1 & \psi_A(q^4x)\\ 0 &\pm1 \end{matrix}\). 
$$

 Performing the matrix multiplication, we obtain the following equation, which holds for every $x\in \Op$ and $q\in \Op^*$:

\begin{equation}
\label{mul-add1}
\psi_1(q)^2\psi_2(q)^2 \psi_A(x) = \psi_A(q^4x).
\end{equation}

By Lemma~\ref{nonreal}, we can find $q\in \Op^*$ so that $q^4$ is a negative integer, say $-r$. Using the fact that additive maps between abelian groups are $\Z$-equivariant, equation (\ref{mul-add1}) becomes 
$$(\psi_1(q)^2\psi_2(q)^2+r)\psi_A(x) = 0.$$
But the above expression is in $\Re$ so that $\psi_1(q)^2\psi_2(q)^2+r$ must be positive. Therefore we must have $\psi_A(x) = 0$ for all $x \in \Op$. This contradicts our assumption that the 1-eigen space of $\f(U_0)$ is 1 dimensional. 
\end{proof}


{\bf Step 2: $ \SL_2(\Op) \to \GL_3(\Re)$}


\begin{proof}
We begin by giving the proof in case $\f: \SL_2(\Op) \to \GL_3(\Re)$ is reducible. We then show that any representation into $\GL_3(\Re)$ must either be reducible or have finite image. 

If $\f$ is reducible, then there is an invariant subspace $V$ of dimension one or two. By extending a basis for $V$ to a basis of $\Re$, we may conjugate with an element of $\GL_3(\Re)$ so that $\f(\SL_2(\Op))$ is an upper block triangular subgroup of $\GL_3(\Re)$. This gives rise to a map from the image $\f(\SL_2(\Op)) \to \GL_1(\Re) \times \GL_2(\Re)$ with abelian kernel. Applying the previously established fact that any representation $\SL_2(\Op) \to \GL_2(\Re)$ has finite image, we see that $\f(\SL_2(\Op))$ contains a finite index abelian subgroup. But, as $\SL_2(\Op)$ is strongly almost perfect (Lemma~\ref{perfect}), we deduce that $\f(\SL_2(\Op))$ is finite.

We now show that either $\f$ is reducible or has finite image. As before, we apply Lemma \ref{NiceImageUnip} to find $U_0$ of finite index in $U_2(\Op)$ so that $\f(U_0)$ is unipotent. 

Let $V_1 \subset \Re^3$ be the 1-eigen space  of $\f(U_0)$. Recall that it is $N_2(\Op)$ invariant since $U_0\norm N$. If $V_1$ is a 3-dimensional space then image of $U_0$ is trivial and hence by Corollary~\ref{FiniteImage}, we get that the image of $SL_2(\Op)$ is finite. If $V_1$ is not 3-dimensional, then either $V_1$ or $\Re^3/V_1$ is two dimensional.  

Again, since $V_1$ is $N_2(\Op)$-invariant, we get two homomorphisms $N_2(\Op) \to \GL(V_1)$ and $N_2(\Op) \to \GL(\Re^3/V_1)$. By Proposition \ref{HomOfN}, we must have that the image of $U_2(\Op)$ in each is finite. In particular, by choice of $V_1$ the image of $U_0$ in both $\GL(V_1)$ and $\GL(\Re^3/V_1)$ is trivial.

Therefore, up to post-composing $\f$ with the transpose inverse automorphism of $\GL_3(\Re)$ if necessary, we may assume that the 1-eigen space of $\f(U_0)$ has dimension two.  

Now, since $U_0$ and $U_0^t$ (the group consisting of transpose matrices of $U_0$) are conjugate inside $\SL_2(\Op)$, the 1-eigen space of the image $\f(U_0^t)$ has dimension two as well. Therefore the intersection of these two 2-dimensional spaces must be non-trivial in $\Re^3$ which means that the image of the group  $\<U_0, U_0^t\>$ has a non-trivial 1-eigen space. 

The group $\<U_0, U_0^t\>$ is of finite index in $\SL_2(\Op)$. Up to passing to a further finite index subgroup if necessary, we may assume that it is normal in $\SL_2(\Op)$ and hence the non-trivial 1-eigen space of this finite index normal subgroup is invariant under $\SL_2(\Op)$. This means that $\f$ is reducible. 

\end{proof}
 
{\bf Step 3: The general case}
\begin{proof}
Now onwards, whenever we speak of $\SL_{n-1}(\Op) \leq \SL_{n}(\Op)$ we mean that we view $\SL_{n-1}(\Op)$ as a subgroup of  $\SL_{n}(\Op)$ embedded in the upper left-hand corner of $\SL_{n}(\Op)$.

To proceed by induction, we assume that the image of any homomorphism $\SL_{n-1}(\Op) \rightarrow \GL_{2n-3}(\Re)$ is finite. By considering $\SL_{n-1}(\Op) \leq \SL_{n}(\Op)$  and using Corollary~\ref{FiniteImage} we get that $\SL_{n}\Op \to \GL_D(\Re)$ has finite image for all $D<2n-3$.

We are left to prove that if $2n-3<D \leq 2n-1$ then the image of $\f: \SL_n(\Op) \rightarrow \GL_{D}(\Re)$ is finite. The following argument works for both $D = 2n-2$ and $D = 2n-1$. The argument for $D = 2n-1$ follows by the induction hypothesis. After proving for $D = 2n-2$, we apply the same argument for $D = 2n-1$ and use the result for $D = 2n-2$ in this.

As before, let $U_0$ be determined by Lemma~\ref{NiceImageUnip}. Let $$L = \{(l_{ij}) \in \SL_n(\Op) \mid l_{ii} = 1, l_{ij} = 0\, \, \forall\, i \neq j\, \mathrm{and} \, j \neq n\}$$ 
be the abelian subgroup of $\SL_n(\Op)$ consisting of matrices having $1$'s on the diagonal and non-trivial entries only in the last column. It is easily verified that $L$ is normalized by $\SL_{n-1}(\Op) \leq \SL_{n}(\Op)$. By intersecting $L$ with $U_0$, we obtain a finite index subgroup $L_0$ of $L$ whose image is unipotent. By Lemma~\ref{finite index subgroup form}, we can pass to a further finite index subgroup and assume that there exists an integer $m$ such that $$L_0 = \{(l_{ij}) \in L \mid l_{ij} \in \pi^m \Op \,\, \forall \,\,  i \neq j \} $$ is contained in $U_0$ and is also normalized by $\SL_{n-1}(\Op)$.

The image  $\f(L_0)$ is unipotent, therefore there exists a flag  
\begin{eqnarray}
\label{JH}
\{0\}=V_0 \subset V_1 \subset V_2 \subset \cdots \subset V_k=\Re^D
\end{eqnarray} of susbpaces of $\Re^D$ with the property that $V_1$ is the maximal 1-eigen space for $\f(L_0)$ and $V_j$ is the maximal 1-eigen space for the quotient action on $\Re^D/V_{j-1}$. Since $\f(L_0)$ is normalized by $\f(\SL_{n-1}(\Op))$, the flag in (\ref{JH}) is preserved by $\f(\SL_{n-1}(\Op))$.

If $k=1$ then $\Re^D$ is the 1-eigenspace of $\f(L_0)$, that is to say the image of $L_0$ is trivial. Therefore the image of $E_{1,n}(\Op) \leq L$ is finite and again, Corollary~\ref{FiniteImage} shows that the image of $\SL_n(\Op)$ is finite.

We now assume that $k>1$. The argument proceeds in two cases depending on whether $2 \leq \dim(V_j) \leq D-2$ for some $j$ or not. Assume that $2 \leq \dim(V_{j_0}) \leq D-2$ for some $j_0$. By assumption on $D$ this means that the dimension, and codimension of $V_{j_0}$ both satisfy the inequality
$$2 \leq \dim(V_{j_0}), D-\dim(V_{j_0})  < 2(n-1).$$
 This now allows us to apply the induction hypothesis to the action of $\f(\SL_{n-1}(\Op))$ on both $V_{j_0}$ and $\Re^D/V_{j_0}$ and we get that the image of the map from $\f(\SL_{n-1}(\Op))$ to $ \GL(V_{j_0}) \times \GL(\Re^D/V_{j_0})$ is finite. Let $\G \leq \SL_{n-1}(\Op)$ be the finite index subgroup with trivial image in $ \GL(V_{j_0}) \times \GL(\Re^D/V_{j_0})$. Since the kernel of the map $ \stab(V_{j_0}) \to \GL(V_{j_0}) \times \GL(\Re^D/V_{j_0})$ is abelian, we see that $\f(\G)$ is abelian, and hence finite since $\SL_{n-1}(\Op)$ is strongly almost perfect (Lemma~\ref{perfect}). In particular, this implies that the image of $E_{12}(\Op)$ is finite, which concludes the proof in this case.
 
Now we are left with the case for which $\dim(V_j) = 1$ or $D-1$ for every $j=1, \dots, k-1$. This means that the flag (\ref{JH}) for $\f(L_0)$ is $\{0\}\subset V_1 \subset \Re^D$, with $V_1$ being either of dimension or codimension one. Again, by postcomposing $\f$ with the transpose inverse automorphism of $\GL_D(\Re)$ if necessary, we can assume that the 1-eigen space of $L_0$ is $D-1$ dimensional. 

Consider the $n$ distinct conjugates of $L$ that correspond to the distinct columns of $\SL_{n}(\Op)$. By taking these conjugates of $L_0$, we generate $\EL_n(\pi^m \Op)$. Each of these column spaces has a $D-1$ dimensional 1-eigenspace, let us call these $W_1, \dots, W_n$. Then, $\Cap{i=1}{n} W_i$ is a 1-eigenspace for $\f(\EL_n(\pi^m \Op))$. The following shows that since $D \geq n+1$, the intersection is not trivial:

\begin{lemma} Let $W_1, W_2, \ldots, W_n$ be co-dimension one subspaces in a $D$ dimensional space. Then $\dim(\cap_{i=1}^n W_i) \geq D-n$.
\end{lemma}
\begin{proof} This result follows by $\dim(W_1 \cup W_2) = \dim(W_1) + \dim(W_2) - \dim(W_1 \cap W_2)$. 
\end{proof}

Let us pass to a finite index subgroup of $\EL_n(\pi^m \Op)$ which is normal in $\SL_{n}(\Op)$. Then $V$, the 1-eigenspace for the image of this subgroup is at least $(D-n)$-dimensional, at most $(D-1)$-dimensional and $\SL_{n}(\Op)$-invariant.

This gives a map $\f(\SL_n(\Op))  \rightarrow \GL(V) \times \GL(\Re^D/V)$. The dimension and co-dimension of $V$ are both less than $D$. We have already established that this means that the image of $\SL_n(\Op)$ in $\GL(V) \times \GL(\Re^D/V)$ is finite(notice that for $D = 2n-2$, we have $\dim(V) \leq 2n-3$ and result follows by induction and for $D = 2n-1$, $\dim(V) \leq 2n-2$ and result follows from $D = 2n-2$). We see that $\f(\SL_n(\Op))$ has to contain a finite index abelian subgroup. But, as $\SL_n(\Op)$ is strongly almost perfect, we deduce that $\f(\SL_n(\Op))$ is finite.

\end{proof}
\begin{cor}
Assume that $|\Op/\pi \Op| > 3$. The image of any representation $\SL_2\Op \to \GL_2 \Re$ is trivial.
\end{cor}

\begin{proof}

Theorem \ref{main theorem} shows that the image of any representation $\SL_2(\Op) \to \GL_2(\Re)$ is finite, therefore compact, and hence contained in a conjugate of the maximal compact subgroup $\mathrm{SO}_2(\Re)$. Since $\mathrm{SO}_2(\Re)$ is abelian and $\SL_2\Op$ is perfect whenever $|\Op/\pi \Op| > 3$, we conclude that the image is trivial. 

\end{proof}


\noindent {\bf Acknowledgements:}
The authors would like to thank Uri Bader for useful comments on a preliminary version of this work and Igor Erovenko for a useful conversation regarding the context of this work. They also would like to thank Technion University for their hospitality as this work was initiated during a workshop hosted there. The first and second listed authors were partially supported by the European Research Council (ERC) grant agreement 203418 and the Center for Advanced Studies in Mathematics at Ben Gurion University respectively.

\bibliography{refs}

\begin{thebibliography}{NZM91}

\bibitem[Bas64]{MR0174604}
H.~Bass.
\newblock {$K$}-theory and stable algebra.
\newblock {\em Inst. Hautes \'Etudes Sci. Publ. Math.}, (22):5--60, 1964.

\bibitem[BHC61]{BorHarChan1}
Armand Borel and Harish-Chandra.
\newblock Arithmetic subgroups of algebraic groups.
\newblock {\em Bull. Amer. Math. Soc.}, 67:579--583, 1961.

\bibitem[BHC62]{BorHarChan2}
Armand Borel and Harish-Chandra.
\newblock Arithmetic subgroups of algebraic groups.
\newblock {\em Ann. of Math. (2)}, 75:485--535, 1962.

\bibitem[BT73]{BorTits}
Armand Borel and Jacques Tits.
\newblock Homomorphismes ``abstraits'' de groupes alg\'ebriques simples.
\newblock {\em Ann. of Math. (2)}, 97:499--571, 1973.

\bibitem[Fer06]{Fer2006}
Talia Fern{\'o}s.
\newblock Relative property ({T}) and linear groups.
\newblock {\em Ann. Inst. Fourier (Grenoble)}, 56(6):1767--1804, 2006.

\bibitem[Hum98]{Humphreys}
James~E. Humphreys.
\newblock {\em {Linear Algebraic Groups}}.
\newblock Springer, 1998.

\bibitem[KS09]{KasSap}
Martin Kassabov and Mark~V. Sapir.
\newblock Nonlinearity of matrix groups.
\newblock {\em J. Topol. Anal.}, 1(3):251--260, 2009.

\bibitem[Mar91]{Margulis}
G.~A. Margulis.
\newblock {\em Discrete subgroups of semisimple {L}ie groups}, volume~17 of
  {\em Ergebnisse der Mathematik und ihrer Grenzgebiete (3) [Results in
  Mathematics and Related Areas (3)]}.
\newblock Springer-Verlag, Berlin, 1991.

\bibitem[NZM91]{MR1083765}
Ivan Niven, Herbert~S. Zuckerman, and Hugh~L. Montgomery.
\newblock {\em An introduction to the theory of numbers}.
\newblock John Wiley \& Sons Inc., New York, fifth edition, 1991.

\bibitem[Rap11]{Rapinchuk}
Igor~A. Rapinchuk.
\newblock On linear representations of {C}hevalley groups over commutative
  rings.
\newblock {\em Proc. Lond. Math. Soc. (3)}, 102(5):951--983, 2011.

\bibitem[She10]{Shen}
Daniel~K. Shenfeld.
\newblock On semisimple representations of universal lattices.
\newblock volume~4, pages 179--193. 2010.

\end{thebibliography}
\bibliographystyle{alpha}

\end{document}